\documentclass[11pt]{amsart}

\usepackage{amsmath, amsthm, latexsym, amssymb, amsfonts, epsf, xcolor}
\usepackage[hidelinks]{hyperref}
\usepackage{caption, subcaption, enumerate, color}
\usepackage{multirow,multicol}
\usepackage{mathtools}
\usepackage{bm}
\usepackage{enumitem}
\usepackage{blkarray}

\usepackage{booktabs}
\usepackage{framed}
\usepackage{setspace}
\usepackage{tikz}
\usepackage{graphicx}
\usepackage{wrapfig}
\usepackage{float}
\usepackage[dvips]{epsfig}
\usepackage{bbm}
\usepackage{url}
\usepackage{booktabs}
\usepackage{bibentry}
\usepackage[normalem]{ulem}
\usepackage{comment}
\usepackage{lingmacros}
\usepackage{cleveref}
\usepackage{extarrows}
\usepackage{multirow}
\usepackage{mathtools}
\usepackage{enumerate}
\usepackage{array}
\usepackage{bigstrut}
\usepackage{exscale,relsize}
\usepackage{graphicx}
\usepackage[utf8]{inputenc}
\usepackage[linesnumbered, ruled]{algorithm2e}
\usepackage{cite}

\newcommand{\F}{{\mathbb F}}

\newcommand{\regt}{\mathtt{reg}}

\numberwithin{equation}{section}
\newtheorem{theorem}{Theorem}[section]
\newtheorem{lemma}[theorem]{Lemma}
\newtheorem{proposition}[theorem]{Proposition}
\newtheorem{corollary}[theorem]{Corollary}

\theoremstyle{definition}
\newtheorem{definition}[theorem]{Definition} 
\newtheorem{remark}[theorem]{Remark}
\newtheorem{example}[theorem]{Example}

\newcommand{\rmv}[1]{}

\DeclareMathOperator{\Dom}{Dom}

\newcommand{\Ass}{\mathtt{Ass}}

\newif\ifcomment
\commentfalse 

\begin{document}


\title[t-graph]{The $\mathrm{v}$-number of $t$-closed neighborhood ideals}

\author[Jaramillo-Velez]{Delio Jaramillo-Velez}
\address[Delio Jaramillo-Velez]{
Departamento de Matemáticas, Estadística e Investigación Operativa.
Instituto Universitario de Matemáticas y Aplicaciones (IMAULL). 
Universidad de La Laguna. La Laguna, Tenerife, Spain
}
\email{djaramil@ull.edu.es}

\author[H. L\'opez]{Hiram H. L\'opez}
\address[Hiram H. L\'opez]{Department of Mathematics\\ Virginia Tech\\ Blacksburg, VA USA}
\email{hhlopez@vt.edu}

\author[Rodrigo San-José]{Rodrigo San-José}
\address[Rodrigo San-José]{Department of Mathematics\\ Virginia Tech\\ Blacksburg, VA USA}
\email{rsanjose@vt.edu}

\author[N. Shalqini]{Nart Shalqini}
\address[Nart Shalqini]{Department of Mathematics\\ Virginia Tech\\ Blacksburg, VA USA}
\email{nart@vt.edu}

\thanks{Hiram H. L\'opez was partially supported by the NSF grant DMS-2502705. Rodrigo San-José was partially supported by the NSF grant DMS-2401558, the Commonwealth Cyber Initiative, an AMS-Simons Travel Grant, and by Grant PID2022-138906NB-C21 funded by MICIU/AEI/10.13039/501100011033 and by ERDF/EU}
\subjclass[2020]{Primary 13F55; Secondary 05C69, 05E40, 13D02}
\keywords{$\mathrm{v}$-number, $t$-closed neighborhood ideals,
square-free monomial ideals, distance domination,
Castelnuovo--Mumford regularity, Hamming graphs}

\begin{abstract}
We give a combinatorial characterization of the $\mathrm{v}$-number of a $t$-closed neighborhood ideal associated to a graph $G$, and investigate its relation with other invariants such as the $t$-domination number and the smallest size of a ball of radius $t$ in the graph $G$. We present several new cases where \cite[Conjecture 4.9]{JLS_v_number_neighborhood} is satisfied; that is, the  $\mathrm{v}$-number is bounded above by the Castelnuovo-Mumford regularity. We analyze the behavior of the $\mathrm{v}$-number of a $t$-closed neighborhood ideal with respect to the parameter $t$. Lastly, we compute explicit formulas for the $\mathrm{v}$-number in the case of paths, cycles, and caterpillars, and analyze the case of Hamming graphs.
\end{abstract}

\maketitle

\section{Introduction} \label{S:intro}
Computing invariants of ideals, such as Krull dimension, projective dimension, or Castelnuovo-Mumford regularity, or constructing ideals with certain properties, is a difficult task, and
by constructing them from graphs, we can obtain families of ideals in which we have combinatorial control over the
invariants and properties ~\cite{moradi2026homological,Villarreal2026,edgeideal}.

This process was first proposed by Rafael H. Villarreal with the initiation of edge ideal theory. The generators of an edge ideal are given by square-free monomials of degree two associated with the edges of the graph. Under this construction, the independence number coincides with the Krull dimension of the quotient~\cite{introedg}, and the Castelnuovo-Mumford regularity is bounded below by the induced matching number of the graph~\cite[Lemma~2.2]{Katzman2006} (see also~\cite{induced}).
 This development has motivated researchers to propose new constructions and to analyze other graph invariants. For instance, path ideals, $t$-clique ideals, binomial edge ideals and closed neighborhood ideals~\cite{conca-negri,t-clique,binomial-edge,def-NG}. The latter captures the notion of domination of the graph through the primary decomposition of the ideal~\cite{def-NG}. 

In this work, we focus on the combinatorial meaning of the Vasconcelos number of the $t$-closed neighborhood ideal associated to $G$. The concept of $t$-closed neighborhood ideals was proposed by Leila Sharifan as an extension of the concept of closed neighborhood ideals by choosing as generators the representation, by a square-free monomial, of the closed neighborhood of each vertex at distance at most $t$. This class of ideals captures the combinatorial notion of distance domination of the graph through their primary decomposition, and the combinatorial meaning of the minimal set of generators, the depth, and Cohen-Macaulayness have been analyzed in the original paper~\cite{intro_t_closed}.

The Vasconcelos number, or just the $\mathrm{v}$-number, is an algebraic invariant of a graded ideal $I$ in a polynomial ring $S:=\mathbb{K}[x_1,\dots,x_n]$ over a field $\mathbb{K}$, defined as
\begin{equation*}
\mathrm{v}(I) := \min \left\{ d \geq 0 \;\middle|\; \exists \ f \in S_d \text{ and } \mathfrak{p} \in \Ass(I) \text{ with } (I:f)=\mathfrak{p} \right\},
\end{equation*}
where $S_d$ denotes the $d$-th graded component of $S$, $\Ass(I)$ denotes the set of associated primes of $I$, and $(I:f)$ is the colon ideal~\cite{intro-vnumber}.

The $\mathrm{v}$-number has served as a source of algebraic estimations of graph invariants thanks to ideals associated with graphs. One of the most well-known of these ideals is the edge ideal $I(G)$~\cite{introedg}. Here, the $\mathrm{v}$-number of $I(G)$ corresponds to the minimal cardinality of an independent set of $G$ such that its neighbors are a minimal vertex cover~\cite{vnumberedge, induced}. Such a characterization provides algebraic conditions to derive conclusions about the graph~\cite{vnumberedge}; for example, an algebraic classification of $W_2$ graphs. For the case of closed neighborhood ideals, the Vasconcelos number estimates the complexity of the external and self private neighbors of the graph $G$~\cite{JLS_v_number_neighborhood}. One of the main goals of this work is to characterize the $\mathrm{v}$-number of $t$-closed neighborhood ideals in terms of minimal $t$-dominating sets and their $t$-private neighbors. We also analyze the relation between the $\mathrm{v}$-number of $t$-closed neighborhood ideals and other invariants, such as the Castelnuovo-Mumford regularity and the smallest size of a ball of radius $t$ in the graph $G$.

The paper proceeds as follows. In Section~\ref{S:Pre} we present preliminary results on commutative algebra and graphs; for more details see~\cite{Villarreal2026}. Section~\ref{S:v-number} is devoted to combinatorial characterization of the Vasconcelos number. In Section~\ref{S:vVSreg} we analyze the relation between the $\mathrm{v}$-number and the Castelnuovo-Mumford regularity. In Section~\ref{S:path-trees-cycles} we focus on the calculation of the $\mathrm{v}$-number of paths, caterpillars, and cycles. In Section~\ref{S:monotonic} we discuss the behavior of the $\mathrm{v}$-number of the $t$-closed neighborhood ideal with respect to the parameter $t$. Lastly, in Section~\ref{S:hamming} we estimate the $\mathrm{v}$-number of the $t$-closed neighborhood ideal of the Hamming graph.


\section{Preliminaries}\label{S:Pre}

In this section, we state some known results and definitions from graph theory and commutative algebra, such as the $t$-power graph, $t$-domination, $t$-closed neighborhood ideals, and Castelnuovo-- Mumford regularity. Throughout the paper, $t$ is a positive integer, $G$ is a finite simple graph, and, unless other notation is specified, we denote $V(G)=\{v_1,\dots,v_n\}$. In \cite{vnumberedge}, the authors prove that the $\mathrm{v}$-number is additive over connected components, and thus we may restrict $G$ to be connected in some of the results. 

\begin{definition}
    Given a graph $G$, the \emph{$t$-closed neighborhood} of a vertex $u \in V(G)$ is the set of all vertices at distance at most $t$ from $u$:
    \begin{equation*}
        N_t[u] = \{\, v \in V(G) : d_G(u, v) \leq t \,\},
    \end{equation*}
where $d_{G}(u,v)$, the graph distance, is the minimum number of edges between $u$ and $v$. The ambient graph is understood from the context. For $U\subseteq V(G)$, we write $N_t[U]=\bigcup_{u\in U}N_t[u]$.
\end{definition}

Next, we introduce a graph construction that is useful in describing the properties of the $t$-closed neighborhood ideals of $G$.

\begin{definition}[\hspace{1pt}\cite{graph_theory}]
  Let $G$ be a connected graph and let $t \geq 1$. The \emph{$t$-th power graph} of $G$, denoted $G^t$, has vertex set $V(G^t)=V(G)$ and edge set
  \begin{equation*}
    E(G^t) = \{\, \{u,v\} : u \neq v \text{ and } d_G(u,v) \leq t \,\},
  \end{equation*}
  that is, we join every pair of distinct vertices that are at distance at most $t$ in $G$.
\end{definition}

\begin{example}
Let $G$ be the path on six vertices. Then $V(G^2)=V(G)$, while $E(G^2)$ consists of the edges of the original path together with the additional edges joining the pairs of vertices at distance exactly $2$ in $G$.

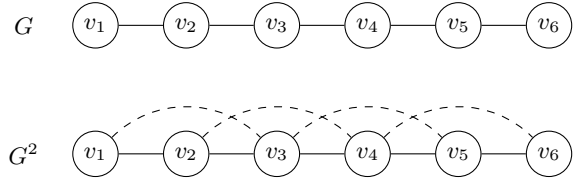
\begin{center}
\begin{tikzpicture}[
  v/.style={circle, draw, fill=white, minimum size=6mm, inner sep=0pt, font=\footnotesize}
]
  \node[font=\footnotesize] at (0.25,0) {$G$};
  \foreach \i in {1,...,6}
    \node[v] (a\i) at (1.2*\i,0) {$v_{\i}$};
  \foreach \i [evaluate=\i as \j using int(\i+1)] in {1,...,5}
    \draw (a\i) -- (a\j);
  \begin{scope}[yshift=-1.7cm]
    \node[font=\footnotesize] at (0.25,0) {$G^{2}$};
    \foreach \i in {1,...,6}
      \node[v] (b\i) at (1.2*\i,0) {$v_{\i}$};
    \foreach \i [evaluate=\i as \j using int(\i+1)] in {1,...,5}
      \draw (b\i) -- (b\j);
    \foreach \i [evaluate=\i as \j using int(\i+2)] in {1,...,4}
      \draw[dashed] (b\i) to[bend left=45] (b\j);
  \end{scope}
\end{tikzpicture}
\end{center}
\end{example}

Next, we define $t$-dominating sets, in analogy with dominating sets; they play a central role in characterizing the $\mathrm{v}$-number of the $t$-closed neighborhood ideals of $G$.

\begin{definition}\label{def:t-domin,t-indep}
     Let $D \subseteq V(G)$. We say that $D$ is a \emph{$t$-dominating set} of $G$ if for every $v \in V(G)$ there exists $u \in D$ such that $v \in N_t[u]$. We say that $U$ $t$-dominates $D$ if $D\subset N_t[U]$.    
     We say $D$ is a $t$-independent set of $G$ if $d_{G}(u,v)>t$ for every pair of distinct vertices $u,v\in D$.
\end{definition}

\begin{example}\label{ex:3_dominating}
    In the graph below, $D = \{c\}$ is a $3$-dominating set of $G$, but it is not a $2$-dominating set of $G$.
\begin{figure}[ht!]
\begin{center}
\begin{tikzpicture}[
  v/.style={circle, draw, fill=white, minimum size=5pt, inner sep=0pt},
  d/.style={v, fill=black}
]
  \node[d, label={[label distance=2pt]below:{\footnotesize $c$}}] (c) at (0,0) {};
  \foreach \a/\name in {90/u, 210/l, 330/r}{
    \node[v] (\name1) at (\a:1) {};
    \node[v] (\name2) at (\a:2) {};
    \node[v] (\name3) at (\a:3) {};
    \draw (c) -- (\name1) -- (\name2) -- (\name3);
  }
\end{tikzpicture}
\end{center}
    \caption{$D=\{c\}$ is a $3$-dominating set of $G$.}
    \label{F:dominating}
\end{figure}
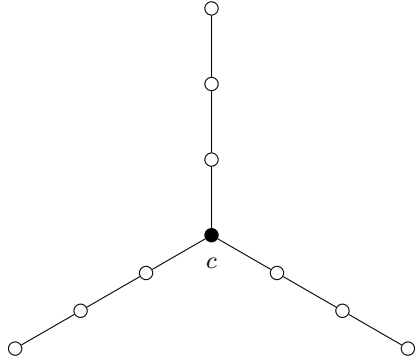

\end{example}

Following the notion of private neighbors~\cite{prmaximizing}, we introduce the notion of $t$-private neighbors (e.g., see \cite{hansberg_distance_domination}). 

\begin{definition}\label{def:t-private}
    Let $G$ be a graph, and $D$ a subset of vertices. A vertex $v \in V(G)$ is a \emph{$t$-private neighbor} of $D$ if $|N_t[v] \cap D| = 1$. Moreover,
\begin{itemize}[label=\textbullet]
    \item if $v \in D$, then $v$ is a \emph{self} $t$-private neighbor of $D$;
    \item if $v \notin D$, then $v$ is an \emph{external} $t$-private neighbor of $D$.
\end{itemize}
We write $P_G^t[D]$ for the set of all $t$-private neighbors of $D$ in $G$.
\end{definition}

Next, we introduce a series of graph invariants which we use to characterize the Vasconcelos number combinatorially.

\begin{definition}[\hspace{1pt}\cite{graph_theory}]\label{def:invariants}
Let $G$ be a graph. We use the following invariants; the first three are standard.
\begin{itemize}
    \item Given a vertex $v\in V(G)$ the eccentricity of $v$ is given by 
    $$\epsilon(v):=\max_{u\in V(G)}d_G(v,u).$$
    \item The radius of $G$ is defined as
    $${\rm rad}(G):=\min_{v\in V(G)}\epsilon(v).$$
    \item The diameter of $G$ is defined as
    $${\rm diam}(G):=\max_{v\in V(G)}\epsilon(v).$$
    \item The smallest volume of a ball of radius $t$ in $G$ is
    $$\delta_t(G) := \min_{v \in V(G)}|N_t[v]|.$$
    \item The $t$-domination number of $G$, $\gamma_{t}(G)$, is the cardinality of a smallest $t$-dominating set.
    \item The $t$-independence number of $G$, $\alpha_{t}(G)$, is the cardinality of a largest $t$-independent set.
    \item The upper domination number of $G$ is defined as
    $$
  \Gamma(G)=\max\{|D|:D\text{ is a minimal dominating set of }G\}.
  $$
\end{itemize}

\end{definition}

We now introduce some algebraic notions, including the $t$-closed neighborhood ideal, Betti numbers, and Castelnuovo-Mumford regularity. 
\begin{definition}[\hspace{1pt}\cite{intro_t_closed}]
For a set of vertices $A$, we let $\mathbf{x}_A = x_1^{c_1}\dots x_n^{c_n}$ be the square-free monomial whose support equals $A$; i.e.,  $c_i = 1$ if and only if $v_i \in A$. Then, the $t$-closed neighborhood ideal of $G$ is defined as 
    \begin{equation*}
        \mathcal{N}_G^t = \langle \mathbf{x}_{N_t[u]} \ : u \in V(G) \rangle,
    \end{equation*}
i.e., it is an ideal generated by square-free monomials parametrized by $t$-closed neighborhoods of all vertices of $G$. It follows that the closed neighborhood ideal $\mathcal{N}_G^1$ introduced in \cite{def-NG} is simply the $1$-closed neighborhood ideal of $G$.
\end{definition}
\begin{example}\label{ex:neighborhood-ideals-path}
Let $G$ be the path on six vertices $v_1,\ldots,v_6$, in this
order. In $S=\mathbb{K}[x_1,\ldots,x_6]$,
where $x_i$ corresponds to $v_i$, we have
\[
\begin{aligned}
\mathcal{N}_G^1
&= \langle x_1x_2,\ x_2x_3x_4,\ x_3x_4x_5,\ x_5x_6\rangle,\\
\mathcal{N}_G^2
&= \langle x_1x_2x_3,\ x_4x_5x_6\rangle,\\
\mathcal{N}_G^3
&= \langle x_1x_2x_3x_4,\ x_3x_4x_5x_6\rangle.
\end{aligned}
\]
All displayed generating sets are minimal. For instance, the
monomial $x_1x_2x_3$ corresponding to $N_1[v_2]$ is redundant
because it is divisible by $x_1x_2$, which corresponds to
$N_1[v_1]$.
\end{example}

For a proper homogeneous ideal $I\subseteq S$, consider the minimal graded free resolution of $S/I$ as an $S$-module:
$$
0 \to \bigoplus_j S(-j)^{b_{g,j}}
\to \cdots \to
\bigoplus_j S(-j)^{b_{1,j}}
\to S \to S/I \to 0.
$$
The numbers $b_{i,j}$ are the \emph{graded Betti numbers} of the minimal resolution of $S/I$, and the \emph{Castelnuovo--Mumford regularity} of $S/I$, or simply the \emph{regularity} of $S/I$, is defined as
$$
\regt(S/I) := \max\{j-i \mid b_{i,j} \neq 0\}.
$$
The integer $g$, denoted by $\mathtt{pd}(S/I)$, is the \emph{projective dimension} of $S/I$; see~\cite{intro-vnumber, Villarreal2026}.

Let $\mathfrak m=\langle x_1,\ldots,x_n\rangle$ be the homogeneous maximal ideal of $S$. The \emph{depth} of $S/I$, denoted by
$\operatorname{depth}(S/I)$, is the maximum length of an
$S/I$-regular sequence in $\mathfrak m$. Here a sequence
$f_1,\ldots,f_r\in\mathfrak m$ is $S/I$-regular if multiplication by $f_i$ is injective on
$S/(I+\langle f_1,\ldots,f_{i-1}\rangle)$ for every
$1\leq i\leq r$; see~\cite[Sections~1.1 and~1.2]{BrunsHerzog1998}.
The Auslander--Buchsbaum formula
\cite[Theorem~1.3.3]{BrunsHerzog1998} gives
\begin{equation}\label{eq:auslander-buchsbaum}
    \operatorname{depth}(S/I)+\mathtt{pd}(S/I)=n.
\end{equation}
We always have $\operatorname{depth}(S/I)\leq\dim(S/I)$, and $S/I$ is called \emph{Cohen--Macaulay} when equality holds; see~\cite[Proposition~1.2.13 and Definition~2.1.1]{BrunsHerzog1998}.

\section{the $\mathrm{v}$-number of $t$-closed neighborhood ideals}\label{S:v-number}

Next, we state a formula for  $\mathrm{v}$-number of $\mathcal{N}^t_{G}$ and we relate this invariant with other algebraic and combinatorial invariants such as the Castelnuovo-Mumford regularity and the radius of a graph.  

\begin{proposition}\label{trivial1}
We have that 
 \begin{equation*}
    \mathcal{N}_G^t = \mathcal{N}_{G^t}^1.
\end{equation*}
\end{proposition}
\begin{proof}
We show that these two ideals have the same generators.    Suppose $\mathbf{x}_A$ is a generator of $\mathcal{N}_{G^t}^1$. Then there exists a vertex $v \in V(G)$ with $A = N_1[v]$ in $G^t$. By construction, $w \in N_1[v]$ in $G^t$ if and only if $d_G(v,w) \leq t$; that is, $N_1[v]$ in $G^t$ equals $N_t[v]$ in $G$. Hence $A = N_t[v]$, so $\mathbf{x}_A$ is a generator of $\mathcal{N}_G^t$. Hence, $\mathcal{N}_{G^t}^1 \subseteq \mathcal{N}_G^t$. The reverse inclusion follows by the same argument.
\end{proof}

\begin{lemma}\label{trivial2}
$D$ is a minimal $t$-dominating set of $G$ if and only if $D$ is a minimal dominating set of $G^t$.
\end{lemma}
\begin{proof}
    By construction, domination in $G^t$ coincides with $t$-domination in $G$, so it remains to compare minimality. Suppose $D$ is not minimal in $G^t$, and let $u \in D$ be redundant, i.e., $D \setminus \{u\}$ still dominates $G^t$. Then $D \setminus \{u\}$ is also a $t$-dominating set of $G$, so $D$ is not a minimal $t$-dominating set of $G$. The converse is analogous.
\end{proof}

\begin{lemma}\label{trivial3}
    We have $P_{G^t}^1[D] = P_G^t[D]$.
\end{lemma}
\begin{proof}
    By Proposition~\ref{trivial1}, $N_1[v]$ in $G^t$ equals $N_t[v]$ in $G$ for every $v \in V(G)$, and therefore $|N_1[v] \cap D| = |N_t[v] \cap D|$. Thus $v$ is a private neighbor of $D$ in $G^t$ if and only if it is a $t$-private neighbor of $D$ in $G$. Since the self/external distinction depends only on whether $v \in D$, it is preserved as well, and hence $P_{G^t}^1[D] = P_G^t[D]$.
\end{proof}

\begin{theorem}\label{P:vformula}
    Let $G$ be a simple, connected graph and let $t \geq 1$. Then
    \begin{equation*}
        \mathrm{v}(\mathcal{N}_G^t) = \min\bigl\{\, |N_t[U] \setminus D| \ : \ D \text{ is a minimal $t$-dominating set of $G$, } U \subseteq P_G^t[D] \text{ $t$-dominates } D \,\bigr\},
    \end{equation*}
    where $N_t[U] = \bigcup_{u \in U} N_t[u]$.
\end{theorem}
\begin{proof}
By~\cite[Thm.~3.6]{JLS_v_number_neighborhood}, for every simple connected graph $H$ we have
    \begin{equation*}
        \mathrm{v}(\mathcal{N}_{H}^1) = \min\bigl\{\, |N_1[U] \setminus D| \ : \ D \text{ is a minimal dominating set of $H$ and } U \subseteq P_{H}^1[D] \text{ dominates } D \,\bigr\}.
    \end{equation*}
    We apply this to $H = G^t$. By Proposition~\ref{trivial1}, $\mathrm{v}(\mathcal{N}_G^t) = \mathrm{v}(\mathcal{N}_{G^t}^1)$, and $N_1[U]$ in $G^t$ equals $N_t[U]$ in $G$ for every $U \subseteq V(G)$. By Lemma~\ref{trivial2}, the minimal dominating sets of $G^t$ are exactly the minimal $t$-dominating sets of $G$; by Lemma~\ref{trivial3}, $P_{G^t}^1[D] = P_G^t[D]$; and $U$ dominates $D$ in $G^t$ if and only if $U$ $t$-dominates $D$ in $G$. Substituting these identifications yields the stated formula.
\end{proof}

Using this result, we relate the $\mathrm{v}$-number to the invariants of the graph $G$ associated to eccentricity, such as the radius and the diameter of $G$.

\begin{lemma}\label{smallestball}
    Let $G$ be a simple connected graph. We have the following lower bound for the $\mathrm{v}$-number of the $t$-closed neighborhood ideal of $G$,
    \begin{equation*}
        \mathrm{v}(\mathcal{N}_G^t) \geq \delta_t(G) - 1.
    \end{equation*}
\end{lemma}
\begin{proof}
    We use the combinatorial description of the $\mathrm{v}$-number from Theorem~\ref{P:vformula}. Assume $U \subseteq P_G^t[D]$ $t$-dominates $D$. For any $u \in U$ we have $|N_t[u] \cap D| = 1$; hence it must be that $|N_t[u] \setminus D|  = |N_t[u]| - |N_t[u] \cap D| \geq \delta_t(G) - 1$, since $|N_t[u]|\geq  \delta_t(G)$ and $|N_t[u] \cap D|= 1$ ($u$ is a $t$-private neighbor). Consequently, $|N_t[U] \setminus D| \geq |N_t[u] \setminus D| \geq \delta_t(G) - 1$. Since $U$ and $D$ were arbitrary, we conclude that the lower bound holds for all such sets, and hence $\mathrm{v}(\mathcal{N}_G^t) \geq \delta_t(G) - 1$. 
\end{proof}

\begin{proposition}\label{prop:v-number-radius}
    Let $G$ be a simple connected graph. We have the following results about the behavior of $\mathrm{v}(\mathcal{N}_G^t)$.
    \begin{itemize}
        \item If ${\rm rad}(G) \leq t $ we have $\mathrm{v}(\mathcal{N}_G^t) =\delta_t(G) -1 $.
        \item If $1 \leq t < {\rm rad}(G)$ we have $t \leq \mathrm{v}(\mathcal{N}_G^t) \leq n - \alpha_t(G)$ where $\alpha_t(G)$ is the $t$-independence number of $G$. 
    \end{itemize}
    In particular, we see that the sequence $\mathrm{v}(\mathcal{N}_G^t)$ is monotonic non-decreasing when $ t \geq {\rm rad}(G)$ and stabilizes at $n-1$.
\end{proposition}
\begin{proof}
    Let ${\rm rad}(G) \leq t$. There exists $u \in V(G)$ such that $N_t[u] = V(G)$, namely at least $u$ such that  $\epsilon(u)={\rm rad}(G)$. Then $D= \{u\}$ is a minimal $t$-dominating set and any of its $t$-neighbors are external $t$-private neighbors. Pick $v \in N_t[u] = V(G)$ with $|N_t[v]| = \delta_t(G)$. Clearly $U = \{v\}$ $t$-dominates $D$ and $ |N_t[U] \setminus D|= |N_t[v]|-1 = \delta_t(G) - 1$. Thus we conclude $\mathrm{v}(\mathcal{N}_G^t) \leq \delta_t(G) - 1$. Equality follows by Lemma \ref{smallestball}. 

    For the second part, note that since $t < {\rm rad}(G)$, for any $u \in V(G)$ there exists $v \in V(G)$ such that $d_G(u,v) = t+1$. Taking the vertices along that path, we see that we have $|N_t[u]| \geq t + 1$. It follows that $\delta_t(G) \geq t + 1$ and hence $\mathrm{v}(\mathcal{N}^t_G) \geq t$ by Lemma \ref{smallestball}. For the upper bound, let $D$ be a maximum $t$-independent set of $G$, so that $|D|=\alpha_t(G)$. Then $D$ is a maximal independent set of $G^t$, and hence a minimal dominating set of $G^t$ by \cite[Cor.~10.37]{chartrand_graphs_digraphs}. By Lemma~\ref{trivial2}, $D$ is a minimal $t$-dominating set of $G$. For $U=D$, we obtain $|N_t[U] \setminus D| \leq |V(G) \setminus D|= n - \alpha_t(G)$. 
\end{proof}

As a corollary, we obtain the $\mathrm{v}$-number for complete $r$-partite graphs, which, for $t=1$, recovers~\cite[Ex.~3.8]{JLS_v_number_neighborhood}, where all parts are assumed to have at least two vertices.

\begin{corollary}
  Let $G=K_{n_1, \dots, n_r}$ be a complete $r$-partite graph with $n_1 \geq n_2 \geq \dots \geq n_r$ and $r\geq 2$. Then $\mathrm{v}(\mathcal{N}_G^1)= \sum_{i=2}^r n_i$ while $\mathrm{v}(\mathcal{N}^t_G)= n -1$ for $t \geq 2$. 
  
  \end{corollary}
  \begin{proof}
      If $t=1$, note that $\delta_1(K_{n_1, \dots, n_r}) = \sum_{i=2}^r n_i +1 =n-n_1+1$ and hence by Lemma \ref{smallestball} we have $\mathrm{v}(\mathcal{N}_{K_{n_1, \dots, n_r}}^1) \geq  \sum_{i=2}^r n_i $. Taking $D=U=V_1$, with $|V_1|=n_1$ (a largest part of the $r$-partite graph) gives the reverse inequality via Theorem \ref{P:vformula}.
      
      If $t \geq 2$, we have $\mathrm{v}(\mathcal{N}^t_{K_{n_1, \dots, n_r}}) = n-1$ since $K_{n_1, \dots, n_r}^t$ is a complete graph for $ t \geq 2$.
\end{proof}

For a nonzero homogeneous ideal $I\subseteq S$, its
\emph{initial degree} is
\[
    \operatorname{indeg}(I)
    := \min\{d\geq 0 : I_d\neq 0\},
\]
where $I_d$ denotes the degree-$d$ homogeneous component of $I$. Equivalently, $\operatorname{indeg}(I)$ is the smallest degree among the minimal homogeneous generators of $I$. For a nonzero proper square-free monomial ideal $I\subseteq S=\mathbb{K}[x_1,\ldots,x_n]$, with square-free monomial generators $m_1,\ldots,m_r$, its \emph{Alexander dual} is the square-free monomial ideal
\[
    I^\vee
    := \bigcap_{j=1}^r
       \langle x_i : x_i \mid m_j\rangle.
\]
Note that redundant generators $m_j$ contribute redundant ideals to the intersection, and thus we can assume $m_1,\dots,m_r$ are minimal generators. The following result connects these concepts for $t$-closed neighborhood ideals with graph invariants. 

\begin{proposition}\label{prop:radius-diameter-detection}
Let $G$ be a connected graph on $n\ge2$ vertices. Then
\[
\operatorname{rad}(G)
=
\min\left\{
t\ge1:
\operatorname{indeg}\big((\mathcal{N}_G^t)^\vee\big)=1
\right\},
\]
while
\[
\operatorname{diam}(G)
=
\min\left\{
t\ge1:
\mathrm{v}(\mathcal{N}_G^t)=n-1
\right\}.
\]
\end{proposition}

\begin{proof}
For $D\subseteq V(G)$, let $\mathbf x_D=\prod_{v_i\in D}x_i$. Since
\[
(\mathcal{N}_G^t)^\vee
=
\bigcap_{v\in V(G)}
(x_i:v_i\in N_t[v]),
\]
we have
\[
\mathbf x_D\in(\mathcal{N}_G^t)^\vee
\Longleftrightarrow
D\cap N_t[v]\neq\emptyset
\text{ for every }v\in V(G).
\]
Thus the minimal generators of $(\mathcal{N}_G^t)^\vee$ correspond to
the minimal $t$-dominating sets of $G$, and hence
\[
\operatorname{indeg}\big((\mathcal{N}_G^t)^\vee\big)=\gamma_t(G).
\]
Therefore,
\[
\operatorname{indeg}\big((\mathcal{N}_G^t)^\vee\big)=1
\Longleftrightarrow
N_t[u]=V(G)\text{ for some }u
\Longleftrightarrow
\operatorname{rad}(G)\le t.
\]
This proves the first equality.

Now suppose $t\ge\operatorname{diam}(G)$. Then
$N_t[u]=V(G)$ for every $u\in V(G)$, so
$
\mathcal{N}_G^t=(x_1\cdots x_n)
$
and consequently
$
\mathrm{v}(\mathcal{N}_G^t)=n-1.
$

Conversely, suppose $t<\operatorname{diam}(G)$. Choose vertices
$u,v$ with $d_G(u,v)>t$ and extend $\{u,v\}$ to a maximal
$t$-independent set $D$. Then $|D|\ge2$, and maximality implies that
$D$ is $t$-dominating (recall \cite[Cor.~10.37]{chartrand_graphs_digraphs}); since $D$ is $t$-independent, it is a minimal $t$-dominating set. Moreover, for $d\in D$, we have 
\(
N_t[d]\cap D=\{d\},
\)
so every element of $D$ is a self $t$-private neighbor of $D$.
Taking $U=D$ in Theorem~\ref{P:vformula} gives
\[
\mathrm{v}(\mathcal{N}_G^t)
\le
|N_t[D]\setminus D|
=
n-|D|
\le n-2.
\]
Hence $\mathrm{v}(\mathcal{N}_G^t)=n-1$ if and only if
$t\ge\operatorname{diam}(G)$, proving the second equality.
\end{proof}

\section{Relation between the $\mathrm{v}$-number and the Castelnuovo--Mumford regularity}\label{S:vVSreg}
In this section, we establish cases of the inequality
\(
\mathrm{v}(\mathcal{N}_G^t)\leq\regt(S/\mathcal{N}_G^t),
\) for $t\geq 1$.
For $t=1$, this is the inequality conjectured in~\cite[Conjecture~4.9]{JLS_v_number_neighborhood}. By Proposition~\ref{trivial1}, $\mathcal{N}_G^t=\mathcal{N}_{G^t}^1$, so results for $t=1$ also apply to general $t$ by replacing $G$ with $G^t$, provided that $G^t$ satisfies the corresponding hypotheses. Throughout, $G$ is a simple connected graph with vertex set $V=\{v_1,\ldots,v_n\}$, where $n\geq2$, and $S=\mathbb K[x_1,\ldots,x_n]$ for a field $\mathbb K$.

The \emph{dominance complex} of $G$ is the simplicial complex
\[
\Delta=\Dom(G)
=\{F\subseteq V:V\setminus F\text{ dominates }G\}.
\]
Its elements are called \emph{faces}, and its inclusion-maximal
faces are called \emph{facets}. For $F\subseteq V$, let
$\mathbf{x}_F=\prod_{v_i\in F}x_i$. The Stanley--Reisner ideal
and Stanley--Reisner ring of $\Delta$ are, respectively,
\[
I_\Delta
=\langle\mathbf{x}_F:F\subseteq V,\ F\notin\Delta\rangle,
\qquad
\mathbb K[\Delta]=S/I_\Delta;
\]
see~\cite[Section~1.1]{MillerSturmfels2005}.
Since $F\notin\Delta$ if and only if $N_1[u]\subseteq F$
for some $u\in V$, we have
\[
I_\Delta=\mathcal{N}_G^1,
\qquad
\mathbb K[\Delta]=S/\mathcal{N}_G^1.
\]
The inequality $\operatorname{depth}(S/\mathcal{N}_G^1)\leq n-\Gamma(G)$ in the next lemma is equivalent to $\Gamma(G)\leq\mathtt{pd}(S/\mathcal{N}_G^1)$, which was proved in~\cite[Proposition~3.5]{moradi2026homological} (see also~\cite[Theorem~11]{CivanSingh2026}); we include a short proof for completeness.

The inequality $\operatorname{depth}(S/\mathcal{N}_G^1)\leq n-\Gamma(G)$ in the next lemma is equivalent to $\Gamma(G)\leq\mathtt{pd}(S/\mathcal{N}_G^1)$, which was proved in~\cite[Proposition~3.5]{moradi2026homological} (see also~\cite[Theorem~11]{CivanSingh2026}); we include a short proof for completeness.

\begin{lemma}\label{lem:v-upper-bound}
Let $R=S/\mathcal{N}_G^1$. Then
\[
\mathrm{v}(\mathcal{N}_G^1)\leq n-\Gamma(G)
\text{ and }
\operatorname{depth} R\leq n-\Gamma(G)\leq\dim R
=n-\gamma_1(G).
\]
\end{lemma}

\begin{proof}
Choose a minimal dominating set $D$ with $|D|=\Gamma(G)$.
For each $d\in D$, minimality gives a vertex $u_d$ such that
$N_1[u_d]\cap D=\{d\}$. Thus $U=\{u_d:d\in D\}$ is contained in
$P_G^1[D]$ and dominates $D$. By Theorem~\ref{P:vformula},
\[
\mathrm{v}(\mathcal{N}_G^1)
\leq |N_1[U]\setminus D|
\leq |V\setminus D|=n-\Gamma(G).
\]
For $D\subseteq V$, set
$\mathfrak p_D=\langle x_i:v_i\in D\rangle$.
The facets of $\Delta$ are precisely the complements of
the minimal dominating sets of $G$. Thus the Stanley--Reisner
decomposition~\cite[Theorem~1.7]{MillerSturmfels2005},
restricted to facets, gives the irredundant prime decomposition
\[
\mathcal{N}_G^1=I_\Delta
=\bigcap_{F\text{ facet of }\Delta}\mathfrak p_{V\setminus F}
=\bigcap_{D\text{ minimal dominating}}\mathfrak p_D.
\]
These primes are precisely the associated primes of $R$.
Hence
\[
\dim R
=\max_{D\text{ minimal dominating}}\dim(S/\mathfrak p_D)
=n-\gamma_1(G).
\]
The depth bound for associated primes
\cite[Proposition~1.2.13]{BrunsHerzog1998} gives
\[
\operatorname{depth} R
\leq\min_{D\text{ minimal dominating}}\dim(S/\mathfrak p_D)
=n-\Gamma(G).
\]
Since $\gamma_1(G)\leq\Gamma(G)$, this proves the remaining
inequalities.
\end{proof}

We next use reduced simplicial homology to compare depth and
regularity. Write $\widetilde H_j(\Delta;\mathbb K)$ for the
$j$-th reduced simplicial homology of $\Delta$ over $\mathbb K$.
Over a field, reduced homology and cohomology have the same
dimensions; see~\cite[Section~1.3]{MillerSturmfels2005}.
For $W\subseteq V$, write
$\Delta[W]=\{F\in\Delta:F\subseteq W\}$ for the restriction
of $\Delta$ to $W$.

\begin{lemma}\label{lem:full-support-homology}
For $R=S/\mathcal{N}_G^1$, there is an integer $i$ such that
$b_{i,n}(R)\neq0$. Consequently,
\[
\operatorname{depth} R\leq\regt(R)\leq\dim R.
\]
\end{lemma}

\begin{proof}
By~\cite[Theorem~1.1]{BrouwerCsorbaSchrijver2009}, the number of dominating sets of $G$ is odd. By taking complements, we can identify these sets with all faces of $\Delta$. Thus, by the definition of the reduced Euler characteristic
\cite[Definition~4.9]{MillerSturmfels2005},
\[
\widetilde\chi(\Delta)
=\sum_{F\in\Delta}(-1)^{|F|-1}
\equiv|\Delta|\equiv1\pmod2.
\]

In particular, $\widetilde\chi(\Delta)\neq0$.
If all reduced homology groups of $\Delta$ over $\mathbb K$
vanished (i.e., if $\Delta$ is acyclic), then~\cite[Lemma~4.10]{MillerSturmfels2005} would give
$\widetilde\chi(\Delta)=0$.
Consequently, over every field $\mathbb K$, there is an integer
$j\geq-1$ such that
$\widetilde H_j(\Delta;\mathbb K)\neq0$.

Hochster's formula~\cite[Corollary~5.12]{MillerSturmfels2005}, summing over all squarefree multidegrees of total degree $q$, and using that reduced homology and cohomology over $\mathbb K$ have the same dimensions, reads
\begin{equation}\label{eq:hochster-graded}
b_{i,q}(R)
=\sum_{\substack{W\subseteq V\\|W|=q}}
\dim_{\mathbb K}\widetilde H_{q-i-1}(\Delta[W];\mathbb K).
\end{equation}
Taking $q=n$ and $i=n-j-1$ gives
$b_{i,n}(R)=\dim_{\mathbb K}\widetilde H_j(\Delta;\mathbb K)\neq0$.
In particular, $i\leq\mathtt{pd} R$. The definition of regularity and~\eqref{eq:auslander-buchsbaum}
now give
\[
\regt(R)\geq n-i\geq n-\mathtt{pd} R
=\operatorname{depth} R.
\]
Finally, $\dim\Delta[W]\leq\dim\Delta$ for every $W\subseteq V$.
Hence~\eqref{eq:hochster-graded} implies
$q-i\leq\dim\Delta+1=\dim R$ whenever $b_{i,q}(R)\neq0$,
proving the upper bound.
\end{proof}

\begin{theorem}\label{thm:main}
If $\mathtt{pd}(S/\mathcal{N}_G^1)=\Gamma(G)$, then
\[
\mathrm{v}(\mathcal{N}_G^1)
\leq n-\Gamma(G)
=\operatorname{depth}(S/\mathcal{N}_G^1)
\leq\regt(S/\mathcal{N}_G^1).
\]
\end{theorem}

\begin{proof}
By~\eqref{eq:auslander-buchsbaum} and the hypothesis,
\[
\operatorname{depth}(S/\mathcal{N}_G^1)
=n-\mathtt{pd}(S/\mathcal{N}_G^1)
=n-\Gamma(G).
\]
The conclusion follows from
Lemmas~\ref{lem:v-upper-bound} and~\ref{lem:full-support-homology}.
\end{proof}

Note that the hypothesis is equivalent to
$\operatorname{depth}(S/\mathcal{N}_G^1)=n-\Gamma(G)$, the largest
depth allowed by its associated primes.
By Proposition~\ref{trivial1}, the theorem also applies to
$\mathcal{N}_G^t$ whenever
$\mathtt{pd}(S/\mathcal{N}_G^t)=\Gamma(G^t)$.

\begin{remark}\label{rem:pd-hypothesis-families}
The hypothesis of Theorem~\ref{thm:main} holds for several
classes of graphs. Indeed, every maximum independent set is a
minimal dominating set, so Lemma~\ref{lem:v-upper-bound}
and~\eqref{eq:auslander-buchsbaum} give
\[
\alpha_1(G)\leq\Gamma(G)
\leq\mathtt{pd}(S/\mathcal{N}_G^1).
\]
In~\cite[Question~3.10]{moradi2026homological}, it is asked whether
$\mathtt{pd}(S/\mathcal{N}_G^1)$ equals $\Gamma(G)$, or even $\alpha_1(G)$,
for chordal and for bipartite graphs. By~\cite[Theorem~2]{CivanSingh2026},
$\mathtt{pd}(S/\mathcal{N}_G^1)=\alpha_1(G)$ whenever $G$
is a K\"onig--Egerv\'ary graph (in particular, a bipartite graph),
a chordal graph, a co-chordal graph, a cograph, or a comparability graph. Hence the hypothesis holds for all these classes,
including paths and caterpillars. For bipartite and chordal graphs, the
inequality $\mathrm{v}(\mathcal{N}_G^1)\leq\regt(S/\mathcal{N}_G^1)$ was
already proved in~\cite[Theorem~4.8]{JLS_v_number_neighborhood}, so
Theorem~\ref{thm:main} extends it to K\"onig--Egerv\'ary graphs,
co-chordal graphs, cographs, and comparability graphs.

Moreover, every power of a tree is chordal;
see~\cite[Section~1]{AgnarssonHalldorsson2007}.
Consequently, Proposition~\ref{trivial1} and
Theorem~\ref{thm:main} (or~\cite[Theorem~4.8]{JLS_v_number_neighborhood}) yield
\[
\mathrm{v}(\mathcal{N}_T^t)
\leq\regt(S/\mathcal{N}_T^t)
\text{ for every tree }T\text{ and every }t\geq1.
\]

For cycles, $\Gamma(C_n)=\lfloor n/2\rfloor$. Indeed, a maximum independent
set shows that $\Gamma(C_n)\geq\lfloor n/2\rfloor$. Conversely, let $D$ be a
minimal dominating set of $C_n$, and map each $d\in D$ to an external private
neighbor of $d$ if $d$ has one; otherwise, by minimality, $d$ has no neighbor
in $D$, and we map $d$ to its clockwise neighbor. This defines an injective map
$D\to V(C_n)\setminus D$, so $|D|\leq n/2$.
Since $\mathcal{N}_{C_n}^1$ is the path ideal of $C_n$ generated by the
monomials corresponding to the paths with three vertices, the
projective-dimension formula in~\cite[Corollary~5.5]{AlilooeeFaridi2015}
shows that the hypothesis holds exactly when $n\not\equiv1\pmod4$.

\end{remark}

\begin{corollary}\label{cor:cm-v-reg}
Let $t\geq1$. If $S/\mathcal{N}_G^t$ is Cohen--Macaulay, then

\[
\mathrm{v}(\mathcal{N}_G^t)
\leq\operatorname{depth}(S/\mathcal{N}_G^t)
=\regt(S/\mathcal{N}_G^t)
=\dim(S/\mathcal{N}_G^t)
=n-\gamma_t(G).
\]
\end{corollary}

\begin{proof}
Set $H=G^t$ and $R=S/\mathcal{N}_H^1=S/\mathcal{N}_G^t$. Since $R$ is Cohen--Macaulay, we have
$\operatorname{depth} R=\dim R$. Thus
Lemma~\ref{lem:v-upper-bound} gives
$n-\Gamma(H)=\operatorname{depth} R=\dim R$.
Theorem~\ref{thm:main} and Lemma~\ref{lem:full-support-homology}
give the desired inequality and equalities.
Finally, $\gamma_1(H)=\gamma_t(G)$ by Lemma~\ref{trivial2}.
\end{proof}

The inequality also holds for $t\geq\operatorname{rad}(G)$,
without any assumption on projective dimension.

\begin{proposition}\label{cor:radius-v-reg}
If $t\geq\operatorname{rad}(G)$, then
\[
\mathrm{v}(\mathcal{N}_G^t)
=\delta_t(G)-1
\leq\regt(S/\mathcal{N}_G^t).
\]
\end{proposition}

\begin{proof}
Proposition~\ref{prop:v-number-radius} gives the equality.
A neighborhood monomial of degree $\delta_t(G)$ is a minimal
generator of $\mathcal{N}_G^t$, since a proper divisor among the
neighborhood generators would have smaller degree.
Thus $b_{1,\delta_t(G)}(S/\mathcal{N}_G^t)\neq0$, and the inequality
follows from the definition of regularity.
\end{proof}

\begin{example}\label{ex:pd-hypothesis-not-cm}
Let $G=P_5$ be a path with vertices $v_1,\ldots,v_5$ in this order. Then
\[
I=\mathcal{N}_G^1
=\langle x_1x_2,\ x_2x_3x_4,\ x_4x_5\rangle.
\]
The Taylor resolution~\cite[Section~4.3.2]{MillerSturmfels2005}
is minimal, since deleting any generator from a nonempty subset
changes its least common multiple. It has the form
\[
0\longrightarrow S(-5)\longrightarrow S(-4)^3
\longrightarrow S(-2)^2\oplus S(-3)
\longrightarrow S\longrightarrow S/I\longrightarrow0.
\]
The minimal dominating sets are
$\{v_1,v_4\}$, $\{v_2,v_4\}$, $\{v_2,v_5\}$, and
$\{v_1,v_3,v_5\}$. Consequently,
\[
\mathtt{pd}(S/I)=3=\Gamma(G),\qquad
\operatorname{depth}(S/I)=\regt(S/I)=2<3=\dim(S/I).
\]
Thus Theorem~\ref{thm:main} gives $\mathrm{v}(I)\leq2$
although $S/I$ is not Cohen--Macaulay and
$1<\operatorname{rad}(G)=2$. One can check that $\mathrm{v}(I)=2$. 
\end{example}


\section{Computations for particular classes of graphs}\label{S:path-trees-cycles}
In this section, we use the combinatorial characterization of the $\mathrm{v}$-number to study it for some classes of graphs, such as paths, caterpillars, and cycles. 
\subsection{Paths} We start the analysis with paths. Given $n\geq 1$, we denote by $P_n$ the path with $n$ vertices. 
\begin{theorem}\label{thm:paths}
Let $t\ge 1$ and $n\ge t+1$. Then
\[
\mathrm{v}(\mathcal{N}^t_{P_n}) \;=\; t\,\gamma_t(P_n) \;=\; t\left\lceil \frac{n}{2t+1}\right\rceil .
\]
\end{theorem}

\begin{proof}
Label the vertices of $P_n$ by $1,\ldots,n$. If $n\leq 2t+1$, then $\operatorname{rad}(P_n)\leq t$, so Proposition~\ref{prop:v-number-radius} gives $\mathrm{v}(\mathcal{N}^t_{P_n})=\delta_t(P_n)-1=t$, since $n\geq t+1$. This agrees with the formula. Hence assume $n>2t+1$, and write $L=2t+1$ and $k=\lceil n/L\rceil\geq 2$. Since every $t$-ball contains at most $L$ vertices, every $t$-dominating set has at least $k$ vertices.

For the upper bound, begin with positions $d_i=t+1+(i-1)L$, for $1\leq i\leq k$, and put $\delta=kL-n$. If $k$ is even, shift $d_k$ to the left by $\delta$; if $k$ is odd, shift both $d_{k-1}$ and $d_k$ to the left by $\delta$. Keep the notation $d_i$ for the resulting positions. Since $0\leq\delta\leq 2t$, these satisfy $d_1=t+1$, $d_k=n-t$, and $1\leq d_{i+1}-d_i\leq 2t+1$. Thus $D=\{d_1,\ldots,d_k\}$ is a $t$-dominating set of minimum cardinality, and $\gamma_t(P_n)=k$.

Set $u_1=1$, and pair the remaining vertices of $D$ consecutively, leaving $d_k$ unpaired when $k$ is even. For each pair $(d_i,d_{i+1})$, we have $d_{i+1}-d_i=2t+1$; choose $u_i=d_i+t$ and $u_{i+1}=d_{i+1}-t$. If $k$ is even, also set $u_k=n$. Each $u_i$ is a $t$-private neighbor of $d_i$: the ball of $u_1$ is $[1,d_1]$, the two balls of a pair are $[d_i,d_{i+1}-1]$ and $[d_i+1,d_{i+1}]$, and the ball of $u_k$ in the even case is $[d_k,n]$. After removing $D$, the left end contributes $t$ vertices, each pair contributes $2t$, and the right end contributes $t$ when present. These contributions are disjoint, so, with $U=\{u_1,\ldots,u_k\}$, Theorem~\ref{P:vformula} gives $\mathrm{v}(\mathcal{N}^t_{P_n})\leq |N_t[U]\setminus D|=tk$.

For the lower bound, let $D=\{d_1<\cdots<d_m\}$ be any minimal $t$-dominating set. In Theorem~\ref{P:vformula}, it suffices to choose one private neighbor $u_i$ for each $d_i$, so that $N_t[u_i]\cap D=\{d_i\}$ (see \cite[Rem. 3.7]{JLS_v_number_neighborhood}). Write $I_i=N_t[u_i]$, $X=\bigcup_{i=1}^m I_i$, and $C=|X\setminus D|$. We have $m\geq k$. Privacy implies that nonconsecutive intervals $I_i$ are disjoint. Moreover, only $I_1$ can contain $1$, and only $I_m$ can contain $n$. Hence $|I_i\setminus D|=2t$ for $1<i<m$, while $|I_1\setminus D|,|I_m\setminus D|\geq t$.

Suppose first that $X\neq V(P_n)$, and choose $x\notin X$. Every $I_i$ lies entirely to one side of $x$. On each side, select alternating balls, starting with the ball nearest $x$. If that side contains an even number $s$ of balls, the selected $s/2$ balls all contribute $2t$ vertices outside $D$. If $s$ is odd, all selected balls except possibly the outermost contribute $2t$, and the outermost contributes at least $t$, giving at least $2t(s-1)/2+t=ts$ vertices. All selected balls are disjoint, including across the two sides. Therefore $C\geq tm\geq tk$.

Finally, suppose $X=V(P_n)$, so that $C=n-m$. Selecting $I_2,I_4,\ldots,I_{m-1}$ when $m$ is odd, or $I_1,I_3,\ldots,I_{m-1}$ when $m$ is even, gives $C\geq t(m-1)$. If $m\geq k+1$, this yields $C\geq tk$. Otherwise $m=k$, and
\[
C=n-k\geq 2t(k-1)\geq tk,
\]
since $n>(k-1)(2t+1)$ and $k\geq 2$. Thus, for every pair of sets $D$ and $U$ such that $U\subseteq P_G^t[D]$ and $U$ $t$-dominates $D$, one has that $|N_t[U]\setminus D|$ is at least $tk$, and Theorem~\ref{P:vformula} gives $\mathrm{v}(\mathcal{N}^t_{P_n})\geq tk$. Together with the upper bound, this proves the result. 

\end{proof}

\subsection{Cycles}
For $n\geq 3$, we denote by $C_n$ the cycle with $n$ vertices.
The following is a well-known result in the literature. We add a proof for completeness.

\begin{proposition}\label{prop:cycle-domination}
Let $n \geq 3$ and $t \geq 1$. The $t$-domination number of a cycle $C_n$ is
\[
    \gamma_t(C_n)=\left\lceil \frac{n}{2t+1}\right\rceil.
\]
\end{proposition}

\begin{proof}
Label the vertices of $C_n$ cyclically by $v_0,\ldots,v_{n-1}$. If $n \leq 2t+1$, then any single vertex $t$-dominates $C_n$, so the result follows. Thus assume $n>2t+1$, and put $k=\lceil n/(2t+1)\rceil$.

Let $D=\{d_0,\ldots,d_{k-1}\}$, where $d_i=v_{i(2t+1)}$. The distance along the cycle between consecutive vertices of $D$ is at most $2t+1$, including the final arc from $d_{k-1}$ to $d_0$. Hence every vertex is within distance $t$ of a vertex of $D$, so $D$ is a $t$-dominating set. On the other hand, every $t$-closed neighborhood contains exactly $2t+1$ vertices. Thus, any $t$-dominating set has at least $\lceil n/(2t+1)\rceil$ vertices, proving the result.
\end{proof}

\begin{lemma}\label{lem:cycle-v-lower}
Let $t \geq 1$ and $n>2t+1$. Then
\[
    \mathrm{v}(\mathcal{N}^t_{C_n}) \geq
    2t\left\lceil \frac{\gamma_t(C_n)}{2}\right\rceil.
\]
\end{lemma}

\begin{proof}
Put $k=\gamma_t(C_n)=\lceil n/(2t+1)\rceil$. By Theorem~\ref{P:vformula}, it suffices to consider an arbitrary minimal $t$-dominating set $D=\{d_0,\ldots,d_{m-1}\}$, written in cyclic order, and a set $U=\{u_0,\ldots,u_{m-1}\}$ such that $N_t[u_i]\cap D=\{d_i\}$ (see \cite[Rem. 3.7]{JLS_v_number_neighborhood}). Write $I_i=N_t[u_i]$ and $X=N_t[U]$. We have $m\geq k$ and $|I_i\setminus D|=2t$ for every $i$.

Each $I_i$ is an arc containing $d_i$ and no other vertex of $D$, so nonconsecutive arcs $I_i$ are disjoint. Choosing alternating arcs around the cycle therefore gives
\[
    |X\setminus D| \geq 2t\left\lfloor\frac{m}{2}\right\rfloor.
\]
This is at least $2t\lceil k/2\rceil$ unless $m=k$ is odd.

Suppose, then, that $m=k$ is odd. If $X\neq V(C_n)$, cut the cycle at a vertex outside $X$ and relabel the vertices of $D$ in the resulting linear order. The arcs $I_i$ become intervals, and $I_0,I_2,\ldots,I_{k-1}$ are pairwise disjoint. Consequently, $|X\setminus D|\geq 2t(k+1)/2=t(k+1)$.

Finally, if $X=V(C_n)$, then $|X\setminus D|=n-k$. Since $n>(k-1)(2t+1)$ and $k\geq 3$, we obtain
\[
    |X\setminus D|=n-k
    \geq 2t(k-1)
    \geq t(k+1).
\]
Thus, for every pair of sets $D$ and $U$ such that $U\subseteq P_G^t[D]$ and $U$ $t$-dominates $D$, one has that $|N_t[U]\setminus D|$ is at least $2t\lceil k/2\rceil$, and Theorem~\ref{P:vformula} gives the desired bound.
\end{proof}

\begin{theorem}\label{thm:cycle-v}
Let $n\geq 3$ and $t\geq 1$. Then
\[
    \mathrm{v}(\mathcal{N}^t_{C_n})=
    \begin{cases}
        2t\left\lceil \dfrac{n}{2(2t+1)}\right\rceil,
        & n>2t+1,\\[6pt]
        n-1, & n\leq 2t+1.
    \end{cases}
\]
\end{theorem}

\begin{proof}
If $n\leq 2t+1$, then every $t$-closed neighborhood is $V(C_n)$, and Proposition~\ref{prop:v-number-radius} gives $\mathrm{v}(\mathcal{N}^t_{C_n})=n-1$. Thus assume $n>2t+1$, and put $k=\gamma_t(C_n)$. By Lemma~\ref{lem:cycle-v-lower}, it remains to prove the upper bound.

Let $D=\{d_0,\ldots,d_{k-1}\}$, where $d_i=v_{i(2t+1)}$, be the minimum $t$-dominating set constructed in Proposition~\ref{prop:cycle-domination}. Let $G_i$ denote the arc strictly between $d_i$ and $d_{i+1}$, with indices read cyclically. In particular, $|G_i|=2t$ for $0\leq i\leq k-2$.

Pair the vertices of $D$ consecutively. For each pair $(d_{2i},d_{2i+1})$, choose $u_{2i}$ to be $t$ steps after $d_{2i}$ and $u_{2i+1}$ to be $t$ steps before $d_{2i+1}$. Then
\[
    N_t[u_{2i}]\setminus D
    =N_t[u_{2i+1}]\setminus D
    =G_{2i}.
\]
If $k$ is odd, choose $u_{k-1}$ to be $t$ steps before $d_{k-1}$, so that $N_t[u_{k-1}]\setminus D=G_{k-2}$. Every $u_i$ is a $t$-private neighbor of $d_i$, and the distinct arcs contributed by these choices are pairwise disjoint. Thus, with $U=\{u_0,\ldots,u_{k-1}\}$, Theorem~\ref{P:vformula} gives
\[
    \mathrm{v}(\mathcal{N}^t_{C_n})
    \leq |N_t[U]\setminus D|
    =2t\left\lceil\frac{k}{2}\right\rceil
    =2t\left\lceil\frac{n}{2(2t+1)}\right\rceil.
\]
Together with Lemma~\ref{lem:cycle-v-lower}, this proves the result.
\end{proof}

\begin{remark}
In particular, for $n\geq t+1$, we have
\[
    \mathrm{v}(\mathcal{N}^t_{C_{2n}})
    =2t\left\lceil\frac{n}{2t+1}\right\rceil.
\]
\end{remark}

\begin{remark}
By Theorem~\ref{thm:paths}, $\mathrm{v}(\mathcal{N}_{P_n}^1)=\gamma_1(P_n)=\lceil n/3\rceil$ for every $n\geq 2$. More generally, let $t\geq 1$, assume $n>2t+1$, and let $k=\left\lceil n/(2t+1) \right\rceil$.
Comparing Theorem~\ref{thm:paths} with Theorem~\ref{thm:cycle-v}, we obtain $\mathrm{v}(\mathcal{N}^t_{P_n})=\mathrm{v}(\mathcal{N}^t_{C_n})$ when $k$ is even,
and $\mathrm{v}(\mathcal{N}^t_{P_n})=\mathrm{v}(\mathcal{N}^t_{C_n})-t$ when $k$ is odd.
\end{remark}

\begin{example}\label{ex:cycle}
Let $n=10$ and $t=1$. Then $k=\gamma_1(C_{10})=4$, and the construction in Theorem~\ref{thm:cycle-v} gives $D=\{v_0,v_3,v_6,v_9\}$ and $U=\{v_1,v_2,v_7,v_8\}$, corresponding to the pairs $(v_0,v_3)$ and $(v_6,v_9)$. Here $N_1[U]\setminus D=\{v_1,v_2,v_7,v_8\}$, while $v_4$ and $v_5$ lie outside $N_1[U]$. Consequently,
\[
    |N_1[U]\setminus D|
    =4
    =2\left\lceil\frac{10}{6}\right\rceil
    =\mathrm{v}(\mathcal{N}^1_{C_{10}}).
\]
\end{example}

\begin{figure}[ht!]
    \centering
    \begin{tikzpicture}
        \draw[line width=8pt, red!15, line cap=round] (90:2.2) arc (90:-18:2.2);
        \draw[line width=8pt, red!15, line cap=round] (-126:2.2) arc (-126:-234:2.2);
        \draw (0,0) circle (2.2);
        \foreach \j/\c/\s in {0/blue/5pt, 1/red/4pt, 2/red/4pt, 3/blue/5pt,
                              4/black/3pt, 5/black/3pt, 6/blue/5pt,
                              7/red/4pt, 8/red/4pt, 9/blue/5pt} {
            \fill[\c] (90-36*\j:2.2) circle (\s);
            \node at (90-36*\j:2.65) {$v_{\j}$};
        }
    \end{tikzpicture}
    \caption{The graph in Example \ref{ex:cycle}.}
    \label{F:c10}
\end{figure}
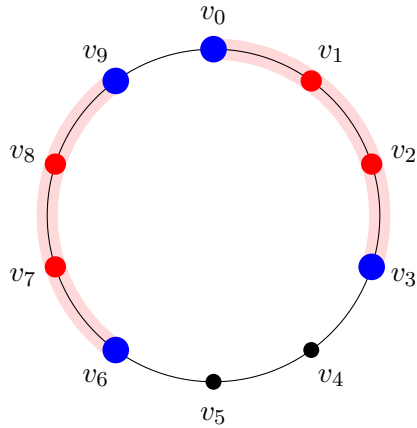

\subsection{Caterpillars}
A tree is called a caterpillar if deleting all its leaves leaves a path (we also admit an empty path in this case). We say that a vertex is a support vertex if it is adjacent to at least one leaf. The spine is the path remaining after deleting all leaves once. For this class, the $\mathrm{v}$-number agrees with the domination number. For the proof of the next result, if \(S\subseteq V(G)\), we write \(G-S\) to denote the induced subgraph obtained by removing the vertices in $S$, and abbreviate \(G-\{v\}\) to \(G-v\). 

\begin{proposition}\label{prop:caterpillar-v}
Let $T$ be a caterpillar with at least two vertices. Then
\[
    \mathrm{v}(\mathcal{N}_T^1)=\gamma_1(T).
\]
\end{proposition}

\begin{proof}
The case $T=P_2$ is immediate. For caterpillars with at least three vertices, we prove a stronger statement by induction on $|V(T)|$: there exists a minimum dominating set $D$ containing every support vertex and a set of external private neighbors $U$ such that
\[
    |U|=|D|
    \qquad\text{and}\qquad
    N_1[U]=D\cup U.
\]
A minimum dominating set can always be chosen to contain every support vertex and no leaf, by replacing any chosen leaf with its neighbor. If the spine of $T$ has at most three vertices, take $D$ to be the set of support vertices and choose one leaf adjacent to each support to form $U$. The set $D$ dominates $T$, and any dominating set must contain at least one vertex from each of the disjoint pairs consisting of a support and its chosen leaf. Thus $D$ is minimum, and the required properties follow.

Now let $(v_0,v_1,\ldots,v_\ell)$ be a longest path of $T$. We may assume $\ell\geq 5$. Note that $v_0$ and $v_\ell$ are leaves, and, thus $v_1,v_{\ell-1}$ are support vertices. Since $T$ is a caterpillar, every vertex outside this path is a leaf. We consider three cases.

\emph{Case 1: $\deg(v_1)>2$.}
Let $T'=T-v_0$. The vertex $v_1$ is still a support vertex in $T'$. Minimum dominating sets of both trees can be chosen to contain $v_1$ and exclude $v_0$, so $\gamma_1(T)=\gamma_1(T')$. Apply the induction hypothesis to $T'$ and obtain $D',U'$. Since $v_1\in D'$ and $U'\cap D'=\varnothing$, adding $v_0$ does not change the closed neighborhood of any vertex of $U'$. Thus $D=D'$ and $U=U'$ have the required properties in $T$.

\emph{Case 2: $\deg(v_1)=2$ and $\deg(v_2)>2$.}
Let $T'=T-\{v_0,v_1\}$. The vertex $v_2$ has a leaf neighbor outside the longest path, so it remains a support vertex in $T'$. We have $\gamma_1(T)=\gamma_1(T')+1$: adding $v_1$ to a minimum dominating set of $T'$ dominates $T$, while deleting $v_1$ from a minimum dominating set of $T$ containing both supports $v_1,v_2$ leaves a dominating set of $T'$.

Apply the induction hypothesis to $T'$ and set $D=D'\cup\{v_1\}$ and $U=U'\cup\{v_0\}$. Since $v_2\in D'$, we have $v_2\notin U'$. Hence every vertex of $U'$ retains its private-neighbor property, while $v_0$ is a private neighbor of $v_1$. Moreover, $N_1[U]=D\cup U$ and $|U|=|D|=\gamma_1(T)$.

\emph{Case 3: $\deg(v_1)=\deg(v_2)=2$.}
Remove the terminal copy of $P_3$ by setting $T'=T-\{v_0,v_1,v_2\}$. We first note that $\gamma_1(T)=\gamma_1(T')+1$. Indeed, adding $v_1$ to any dominating set of $T'$ dominates $T$. Conversely, take a minimum dominating set of $T$ containing $v_1$ and excluding $v_0$. Removing $v_1$ and replacing $v_2$ with $v_3$ whenever $v_2$ is present gives a dominating set of $T'$ with at most $\gamma_1(T)-1$ vertices.

Apply the induction hypothesis to $T'$. Put $D=D'\cup\{v_1\}$ and define
\[
    U=
    \begin{cases}
        U'\cup\{v_0\}, & v_3\notin U',\\
        U'\cup\{v_2\}, & v_3\in U'.
    \end{cases}
\]
If $v_3\notin U'$, the closed neighborhoods of the vertices of $U'$ are unchanged, and $v_0$ is a private neighbor of $v_1$. Thus $N_1[U]=D\cup U$.

If $v_3\in U'$, then $v_3\notin D'$ since every vertex of $U'$ is an external private neighbor of $D'$. Consequently, $N_1[v_2]\cap D=\{v_1\}$, so $v_2$ is a private neighbor of $v_1$ (recall that $\deg(v_2)=2$, i.e., $N_1[v_2]=\{v_1,v_2,v_3\}$). Every vertex of $U'$ retains its private-neighbor property, and extending from $T'$ to $T$ adds only $v_2$ to $N_1[U']$. Since $N_1[v_2]=\{v_1,v_2,v_3\}$ and $v_3\in U'$, we again obtain $N_1[U]=D\cup U$. In both subcases, $|U|=|D|=\gamma_1(T)$.

In each case, $D$ also contains every support vertex of $T$, completing the induction. Theorem~\ref{P:vformula} now gives
\[
    \mathrm{v}(\mathcal{N}_T^1)\leq |N_1[U]\setminus D|
    =|U|
    =\gamma_1(T).
\]
The reverse inequality follows from \cite[Thm. 4.1]{JLS_v_number_neighborhood}.
\end{proof}

\begin{example}\label{ex:tree-counterexample}
The equality $\mathrm{v}(\mathcal{N}_T^1)=\gamma_1(T)$ does not hold for every tree. Let $T$ consist of the path $(x,r,y)$ and four disjoint paths $(s_i,q_i,a_i,b_i,c_i,d_i)$, for $1\leq i\leq 4$, with the additional edges $xq_1,xq_2,yq_3,yq_4$. Then $T$ has $27$ vertices and one can check (for example, using the Sage functions \cite{sagemath} from \cite{JLS_v_number_neighborhood}) that $\gamma_1(T)=9$ and $\mathrm{v}(\mathcal{N}_T^1)=10$.
\end{example}

\section{Monotonic behavior of the $\mathrm{v}$-number}\label{S:monotonic}

The asymptotic behavior of algebraic invariants along powers of an ideal is a classical theme in commutative algebra; for instance, associated primes and depth eventually stabilize, while Castelnuovo--Mumford regularity is eventually linear \cite{Brodmann1979,CutkoskyHerzogTrung1999,Kodiyalam2000,brodmann_asymptotic_spread}. 

More recently, Ficarra and Sgroi proved that the $\mathrm{v}$-number of ordinary powers is eventually linear \cite{FicarraSgroi2026}. In this section we consider an analogous problem, but studying the family $t \mapsto \mathrm{v}(\mathcal{N}_G^t)$, i.e., studying the powers of the graph instead of the powers of the ideal. This family need not be monotonic. Consider the two examples below.
\begin{example}
    Consider the path with six vertices $P_6$. A quick calculation reveals the following sequence for the $\mathrm{v}$-number
    \begin{equation*}
       \{\mathrm{v}(\mathcal{N}^t_{P_6})\}^6_{t=1} =  (2, 4, 3, 4, 5, 5).
    \end{equation*}
\end{example}
\begin{example}
For $C_{15}$, the $\mathrm{v}$-number sequence is
\begin{equation*}
    \{\mathrm{v}(\mathcal{N}^t_{C_{15}})\}^{15}_{t=1} =  (6, 8, 12, 8, 10, 12, 14, \dots, 14).
\end{equation*}
\end{example}

One could wonder if the minimum for the $\mathrm{v}$-number occurs at $t=1$ or $t = \operatorname{rad}(G)$. As we see below, this is not the case either.

\begin{example}
    Consider a path of $6$ vertices with two additional vertices attached to the middle vertices (see figure below).

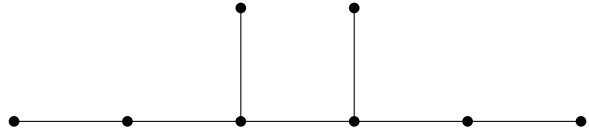
\begin{center}
\begin{tikzpicture}[scale=1]
  \coordinate (x1) at (0,0);
  \coordinate (x2) at (1.5,0);
  \coordinate (x3) at (3,0);
  \coordinate (x4) at (4.5,0);
  \coordinate (x5) at (6,0);
  \coordinate (x6) at (7.5,0);
  \coordinate (y)  at (3,1.5);
  \coordinate (z)  at (4.5,1.5);

  \draw (x1)--(x2)--(x3)--(x4)--(x5)--(x6);
  \draw (x3)--(y);
  \draw (x4)--(z);

  \fill (x1) circle (2pt);
  \fill (x2) circle (2pt);
  \fill (x3) circle (2pt);
  \fill (x4) circle (2pt);
  \fill (x5) circle (2pt);
  \fill (x6) circle (2pt);
  \fill (y) circle (2pt);
  \fill (z) circle (2pt);
\end{tikzpicture}
\end{center}

A quick calculation leads to the following sequence of the $\mathrm{v}$-number for the given graph: 
\begin{equation*}
    \{\mathrm{v}(\mathcal{N}^t_{G})\}^{5}_{t=1} =  (4, 3, 4, 6, 7).
\end{equation*}
As a result we see that the minimum is attained neither at $t=1$ nor at $ t = \operatorname{rad}(G) = 3$. 
\end{example}
For certain families of graphs, such as paths and cycles, we can conclude that the minimum is indeed attained at $t=1$. By the previous example, this is not true for caterpillars.

\begin{proposition}
Let $n\geq 2$. The function $t\mapsto \mathrm{v}(\mathcal{N}_{P_n}^t)$, for $t\geq 1$, attains its minimum at $t=1$.
\end{proposition}

\begin{proof}
We may assume that $t\leq n-1$, since otherwise
\(
\mathrm{v}(\mathcal{N}_{P_n}^t)
=n-1
\geq \left\lceil\frac{n}{3}\right\rceil
=\mathrm{v}(\mathcal{N}_{P_n}^1).
\)
By Theorem~\ref{thm:paths}, we have
\[
\mathrm{v}(\mathcal{N}_{P_n}^t)
=t\left\lceil\frac{n}{2t+1}\right\rceil
\geq \frac{tn}{2t+1}
\geq \frac{n}{3},
\]
where the last inequality follows from $t\geq 1$.
Since the $\mathrm{v}$-number is an integer, it follows that
\(
\mathrm{v}(\mathcal{N}_{P_n}^t)
\geq \left\lceil\frac{n}{3}\right\rceil
=\mathrm{v}(\mathcal{N}_{P_n}^1),
\)
as desired.
\end{proof}

\begin{proposition}
Let $n\geq 3$. The function $t\mapsto \mathrm{v}(\mathcal{N}_{C_n}^t)$, for $t\geq 1$, attains its minimum at $t=1$.
\end{proposition}

\begin{proof}
If $n=3$, then $\mathrm{v}(\mathcal{N}_{C_n}^t)=2$ for every $t\geq 1$ by Theorem~\ref{thm:cycle-v}. Hence assume $n\geq 4$, so that $\mathrm{v}(\mathcal{N}_{C_n}^1)=2\lceil n/6\rceil$ by Theorem~\ref{thm:cycle-v}. If $n\leq 2t+1$, then $\mathrm{v}(\mathcal{N}_{C_n}^t)=n-1\geq 2\lceil n/6\rceil$, since $n\geq 4$. If $n>2t+1$, then, since $t\lceil x\rceil\geq\lceil tx\rceil$ for every real number $x$ and $\frac{tn}{2(2t+1)}\geq\frac{n}{6}$ for $t\geq 1$, Theorem~\ref{thm:cycle-v} gives
\[
\mathrm{v}(\mathcal{N}_{C_n}^t)
=2t\left\lceil\frac{n}{2(2t+1)}\right\rceil
\geq 2\left\lceil\frac{tn}{2(2t+1)}\right\rceil
\geq 2\left\lceil\frac{n}{6}\right\rceil
=\mathrm{v}(\mathcal{N}_{C_n}^1).
\]
\end{proof}

\section{Hamming Graphs}\label{S:hamming}

Let $q$ be a prime power, and consider the space $\mathbb{F}_q^n$ with the Hamming distance. Given two points $y,z\in \mathbb{F}_q^n$, we consider the Hamming distance $d(y,z)=|\{j:y_j\neq z_j\}|$. The associated Hamming graph $H(n,q)$ has the points of $\mathbb{F}_q^n$ as vertices, and there is an edge between two vertices if the Hamming distance between the corresponding points is equal to one. Note that in this section $n$ denotes the length of the words, so $H(n,q)$ has $q^n$ vertices. A connection between the $\mathrm{v}$-number of closed neighborhood ideals and Hamming graphs was given in \cite{JLS_v_number_neighborhood}. 
For $0\le r\le n$, let
\[
V_q(n,r)=\sum_{j=0}^{r}\binom{n}{j}(q-1)^j
\]
 denote the cardinality of a ball of radius $r$ in $H(n,q)$. We have the following bounds for the $\mathrm{v}$-number, in terms of the $t$-domination number. These bounds may not always be stronger than the bound in Lemma \ref{smallestball}, which can also be taken into account to lower bound the $\mathrm{v}$-number.

\begin{theorem}\label{thm:improved-hamming-bound}
Let $2\le t\le n$. For every integer $1\le a\le t-1$, set
\[
r_a=\left\lfloor\frac{t+1-a}{2}\right\rfloor.
\]
Then
\[
\mathrm{v}\big(\mathcal{N}^t_{H(n,q)}\big)
\ge
\min\big\{V_q(n,r_a)-1,\,2^a-1\big\}
\cdot\gamma_t(H(n,q)).
\]
Consequently,
\[
\mathrm{v}\big(\mathcal{N}^t_{H(n,q)}\big)
\ge
\max_{1\le a\le t-1}
\min\big\{V_q(n,r_a)-1,\,2^a-1\big\}
\cdot\gamma_t(H(n,q)).
\]
\end{theorem}

\begin{proof}
Let $D$ be a minimal $t$-dominating set of $H(n,q)$, and let
$U\subseteq P^t_{H(n,q)}[D]$ $t$-dominate $D$. By \cite[Rem. 3.7]{JLS_v_number_neighborhood}, it suffices to consider $U$ such that $|U| = |D|$. Thus let
$
U=\{u_c:c\in D\}$ with $
N_t[u_c]\cap D=\{c\}$ i.e., we index the elements of $U$ by the vertices that $t$-dominate in $D$. 
Set
\[
s_c:=d(u_c,c).
\]
For distinct $c,c'\in D$, we have
$
s_c\le t,
$ and
$d(u_c,c')\ge t+1.
$
By the triangle inequality, we obtain 
$
d(c,c')
\ge d(u_c,c')-d(u_c,c)
\ge t+1-s_c.
$
Interchanging $c$ and $c'$ gives
\begin{equation}\label{eq:hamming-separation}
d(c,c')\ge t+1-\min\{s_c,s_{c'}\}.
\end{equation}

Let $
J_c=\{j\in\{1,\ldots,n\}:(u_c)_j\neq c_j\}$
be the support of the Hamming difference between $u_c$ and $c$. By definition, we have  $|J_c|=s_c$. For $A\subseteq J_c$, let $h_c(A)\in \F_q^n$ be the word
which agrees with $c$ on the coordinates in $A$ and with $u_c$ on
the remaining coordinates, and set
\[
H_c=\{h_c(A):A\subseteq J_c\}.
\]
The words in $H_c$ are distinct, so
$
|H_c|=2^{s_c}.
$
Moreover, for every $h\in H_c$,
\begin{equation}\label{eq:hamming-geodesic}
d(h,c)+d(h,u_c)=s_c.
\end{equation}

For every $c\in D$, define
\[
B_c=
\begin{cases}
N_{r_a}[c],&s_c\le a-1,\\[1mm]
H_c,&s_c\ge a.
\end{cases}
\]
We first observe that
$
B_c\subseteq N_t[u_c].
$
Indeed, if $s_c\le a-1$ and $x\in N_{r_a}[c]$, then
\[
d(x,u_c)
\le d(x,c)+d(c,u_c)
\le r_a+a-1
\le t.
\]
If $s_c\ge a$ and $h\in H_c$, then by
\eqref{eq:hamming-geodesic} we have
$
d(h,u_c)\le s_c\le t.
$

Since $N_t[u_c]\cap D=\{c\}$ and $c\in B_c$, it follows that
\begin{equation}\label{eq:hamming-intersection-D}
B_c\cap D=\{c\}.
\end{equation}

We claim that the sets $B_c$, for $c\in D$, are pairwise disjoint.
Let $c,c'\in D$ be distinct. We consider three cases.

\medskip

\noindent
\textit{Case 1.} Suppose $s_c,s_{c'}\le a-1$. By
\eqref{eq:hamming-separation} we have
$
d(c,c')\ge t-a+2.
$
On the other hand,
\[
2r_a\le t+1-a<t-a+2.
\]
Therefore
$
N_{r_a}[c]\cap N_{r_a}[c']=\emptyset.
$

\medskip

\noindent
\textit{Case 2.} Suppose $s_c,s_{c'}\ge a$. Assume that
$h\in H_c\cap H_{c'}$. Set
$
b=d(h,c),
$ and $
b'=d(h,c').$
By \eqref{eq:hamming-geodesic} we have $
d(h,u_c)=s_c-b$ and 
$
d(h,u_{c'})=s_{c'}-b'.$
Since $u_c$ is $t$-private to $c$ and $u_{c'}$ is $t$-private to
$c'$, we obtain
\[
t+1
\le d(u_c,c')
\le d(u_c,h)+d(h,c')
=s_c-b+b'
\]
and, symmetrically,
\[
t+1
\le d(u_{c'},c)
\le s_{c'}-b'+b.
\]
Adding these inequalities gives
\[
2t+2\le s_c+s_{c'}\le2t,
\]
a contradiction.

\medskip

\noindent
\textit{Case 3.} Suppose, after possibly exchanging $c$ and $c'$,
that $
s_c\ge a $ and $
s_{c'}\le a-1.$
Assume that $h\in H_c\cap N_{r_a}[c']$, and put
$
b=d(h,c).
$
Using \eqref{eq:hamming-geodesic}, we have
\[
t+1
\le d(u_c,c')
\le d(u_c,h)+d(h,c')
\le s_c-b+r_a.
\]
Therefore
\[
b\le s_c+r_a-t-1\le r_a-1.
\]
It follows that
\[
d(c,c')
\le d(c,h)+d(h,c')
\le b+r_a
\le2r_a-1
\le t-a.
\]
However, \eqref{eq:hamming-separation} gives
\[
d(c,c')\ge t+1-s_{c'}\ge t-a+2,
\]
which is a contradiction.

Thus, the sets $B_c$ are pairwise disjoint. If $s_c\le a-1$, then
by \eqref{eq:hamming-intersection-D},
\[
|B_c\setminus D|
=
|N_{r_a}[c]|-1
=
V_q(n,r_a)-1.
\]
If $s_c\ge a$, then
\[
|B_c\setminus D|
=
|H_c|-1
=
2^{s_c}-1
\ge2^a-1.
\]
Since the sets $B_c$ are pairwise disjoint and
$B_c\subseteq N_t[u_c]$, we conclude that
\[
\begin{aligned}
|N_t[U]\setminus D|
&\ge
\sum_{c\in D}|B_c\setminus D|\\
&\ge
\min\big\{V_q(n,r_a)-1,\,2^a-1\big\}|D|\\
&\ge
\min\big\{V_q(n,r_a)-1,\,2^a-1\big\}
\gamma_t(H(n,q)).
\end{aligned}
\]
The first inequality now follows from Theorem~\ref{P:vformula}. Since the
inequality holds for every $1\le a\le t-1$, taking the maximum over
$a$ gives the second inequality.
\end{proof}

\begin{remark}\label{rem:perfect-codes}
A code $C\subseteq\mathbb F_q^n$ is $t$-perfect if the balls $N_t[c]$, for $c\in C$, partition $\mathbb F_q^n$. Hence $C$ is a $t$-dominating set and $|C|V_q(n,t)=q^n$. On the other hand, every $t$-dominating set $D$ satisfies $q^n\leq |D|V_q(n,t)$. Therefore,
\[
\gamma_t(H(n,q))=|C|=\frac{q^n}{V_q(n,t)}.
\]
For $1\leq t<n$, the possible parameters of $t$-perfect codes are precisely those of the following codes~\cite{existing_perfe}: the $q$-ary Hamming codes $[n,n-r,3]_q$, where $n=(q^r-1)/(q-1)$ and $r\geq2$, for $t=1$; the binary repetition codes $[2t+1,1,2t+1]_2$; and the Golay codes $[23,12,7]_2$ and $[11,6,5]_3$, for $t=3$ and $t=2$, respectively. Thus the Hamming codes give $\gamma_1(H(n,q))=q^{n-r}$, while
\[
\gamma_t(H(2t+1,2))=2,\qquad
\gamma_3(H(23,2))=2^{12},\qquad
\gamma_2(H(11,3))=3^6.
\]
For $t=n$, codes with a single codeword are perfect and
$\gamma_n(H(n,q))=1$.
\end{remark}

\begin{remark}
Taking $a=t-1$ gives $r_a=1$ and
\[
V_q(n,1)-1=n(q-1).
\]
Hence Theorem~\ref{thm:improved-hamming-bound} gives the simpler bound
\[
\mathrm{v}\big(\mathcal{N}^t_{H(n,q)}\big)
\ge
\min\{n(q-1),\,2^{t-1}-1\}
\gamma_t(H(n,q)).
\]
The improvement can be strict. For example, let $n=20$, $q=2$, and
$t=8$. Taking $a=5$ gives $r_a=2$ and
\[
\min\{V_2(20,2)-1,\,2^5-1\}
=
\min\{210,31\}
=
31,
\]
whereas the choice $a=t-1$ gives the coefficient $20$.
\end{remark}

\begin{proposition}\label{prop:hamming-radius-n-minus-one}
Let $n\ge2$ and $q\ge2$. Then
\[
\mathrm{v}\big(\mathcal{N}^{n-1}_{H(n,q)}\big)=q^n-q.
\]
\end{proposition}

\begin{proof}
For $1\le j\le n$ and $D\subseteq\mathbb F_q^n$, let $
\pi_j(D)=\{c_j:c\in D\}
$ be the set of $j$th coordinates of $D$. A word $x\in\mathbb F_q^n$ is not $(n-1)$-dominated by $D$ if and
only if
$
d(x,c)=n
$
for every $c\in D$. This is equivalent to
$
x_j\notin\pi_j(D)
$
for every $1\le j\le n$. Therefore, $D$ is an $(n-1)$-dominating set
if and only if
\begin{equation}\label{eq:full-projection}
\pi_j(D)=\mathbb F_q
\end{equation}
for some coordinate $j$.

Let $D$ be a minimal $(n-1)$-dominating set, and fix a coordinate
$j$ satisfying \eqref{eq:full-projection}. For every $c\in D$, the
set $D\setminus\{c\}$ is not $(n-1)$-dominating. Hence
$
\pi_j(D\setminus\{c\})\neq\mathbb F_q.$
It follows that $c_j$ does not occur as the $j$-th coordinate of any
other element of $D$. Thus, the map
\[
D\longrightarrow\mathbb F_q,
\qquad
c\longmapsto c_j,
\]
is both injective and surjective. Consequently,
\begin{equation}\label{eq:minimal-hamming-size}
|D|=q.
\end{equation}

Let $U=\{u_c:c\in D\}\subseteq P^{n-1}_{H(n,q)}[D]$ $(n-1)$-dominate
$D$, where
$
N_{n-1}[u_c]\cap D=\{c\}.
$
For every $c'\in D\setminus\{c\}$, privacy gives
$
d(u_c,c')=n.
$
In particular,
$
(u_c)_j\neq c'_j
$
for every $c'\neq c$. By \eqref{eq:full-projection} and
\eqref{eq:minimal-hamming-size}, the set
\[
\{c'_j:c'\in D\setminus\{c\}\}
\]
contains every element of $\mathbb F_q$ except $c_j$. Therefore
$
(u_c)_j=c_j.
$
It follows that
$
\pi_j(U)=\mathbb F_q,
$
so $U$ is itself an $(n-1)$-dominating set. Hence
\[
N_{n-1}[U]=V(H(n,q))
\]
and
\[
|N_{n-1}[U]\setminus D|
=
q^n-|D|
=
q^n-q.
\]
Since $(D,U)$ was an arbitrary pair such that $U\subseteq P_G^t[D]$ and $U$ $t$-dominates $D$, Theorem~\ref{P:vformula}
gives
\[
\mathrm{v}\big(\mathcal{N}^{n-1}_{H(n,q)}\big)\ge q^n-q.
\]

For the reverse inequality, consider
\[
D=\{(a,a,\ldots,a):a\in\mathbb F_q\}.
\]
The set $D$ is $(n-1)$-dominating, and any two distinct elements of
$D$ are at distance $n$. Hence every $c\in D$ is a self
$(n-1)$-private neighbor of $D$. Taking $U=D$, Theorem~\ref{P:vformula} gives
\[
\mathrm{v}\big(\mathcal{N}^{n-1}_{H(n,q)}\big)
\le
|N_{n-1}[D]\setminus D|
=
q^n-q.
\]
Combining the two inequalities proves the result.
\end{proof}

\section*{Declarations}
\subsection*{Conflict of interest} The authors declare no conflict of interest.

\bibliographystyle{abbrv}
\bibliography{bibliography}

\end{document}